\documentclass{amsart}
\usepackage{amssymb}
\usepackage{amsmath}
\usepackage{amsfonts}

\newtheorem{theorem}{Theorem}
\theoremstyle{plain}

\newtheorem{definition}{Definition}

\newtheorem{lemma}{Lemma}

\newtheorem{proposition}{Proposition}

\numberwithin{equation}{section}
\input{tcilatex}

\begin{document}
\title[On uniqueness of solutions of Navier-Stokes equations]{On uniqueness
of weak solutions of incompressible Navier-Stokes equations in 3-dimensional
case}
\author{Kamal N. Soltanov}
\address{{\small Institute of Mathematics and Mechanics National Academy of
Sciences of Azerbaijan, AZERBAIJAN; Department of Mathematics, Faculty of
Sciences, Hacettepe University, Ankara, TURKEY}}
\email{soltanov@hacettepe.edu.tr ; sultan\_kamal@hotmail.com}
\urladdr{http://www.mat.hacettepe.edu.tr/personel/akademik/ksoltanov/index.html}
\date{}

\begin{abstract}
Here we investigate 3-dimensional Navier-Stokes Equations in the
incompressible case with use of different approach and we prove the
uniqueness of the weak solutions for the data from the space, which is dense
in usual space of data. Moreover we study the solvability and uniqueness of
the weak solutions of problems associated with investigation of the main
problem.
\end{abstract}

\maketitle

\section{\label{Sec_1}Introduction}

In this article we investigate Navier-Stokes equations in the incompressible
case, i.e. we consider the following system of equations: 
\begin{equation}
\frac{\partial u_{i}}{\partial t}-\nu \Delta u_{i}+\underset{j=1}{\overset{d}{\sum }}u_{j}\frac{\partial u_{i}}{\partial x_{j}}+\frac{\partial p}{\partial x_{i}}=h_{i},\quad i=\overline{1,d},
\label{1}
\end{equation}%
\begin{equation}
\func{div}u=\underset{i=1}{\overset{d}{\sum }}\frac{\partial u_{i}}{\partial x_{i}}=0,\quad x\in \Omega \subset R^{d},t>0\quad ,
\label{2}
\end{equation}%
\begin{equation}
u\left( 0,x\right) =u_{0}\left( x\right) ,\quad x\in \Omega ;\quad u\left\vert \ _{\left( 0,T\right) \times \partial \Omega }\right. =0
\label{3}
\end{equation}%
where $\Omega \subset R^{d}$ is a bounded domain with sufficiently smooth
boundary $\partial \Omega $, $T>0$ is a positive number. As it is well known
Navier-Stokes equations describe the motion of a fluid in $R^{d}$ ($d=2$ or $%
3$). These equations are to be solved for an unknown velocity vector $%
u(x,t)=\left\{ u_{i}(x,t)\right\} _{1}^{d}\in R^{d}$ and pressure $p(x,t)\in
R$, defined for position $x\in R^{d}$ and time $t\geq 0$, $h_{i}(x,t)$ are
the components of a given, externally applied force (e.g. gravity), $\nu $
is a positive coefficient (the viscosity), $u_{0}\left( x\right) \in R^{d}$
is a sufficiently smooth vector function (vector field).

As known in \cite{Ler1} is shown (see, also, \cite{Lad1}, \cite{MajBer}, 
\cite{Con1}, \cite{Fef1}, \cite{Lio1}) that the Navier--Stokes equations (%
\ref{1}), (\ref{2}), (\ref{3}) in three dimensions always have a weak
solution $(p,u)$ with suitable properties. But the uniqueness of weak
solutions of the Navier--Stokes equation is not known in three space
dimensions case. Uniqueness of weak solution were proved in two space
dimensions case (\cite{LioPro}, \cite{Lio1}, see also \cite{Lad2}), and
under complementary conditions on smoothnes of the solution three dimensions
case was studied also (see, for example, \cite{Lio1}, \cite{Fur}, etc.).\
For the Euler equation, uniqueness of weak solutions is strikingly false
(see, \cite{Sch1}, \cite{Shn1}).

It is needed to note that the regularity of solutions in three dimensions
case were investigated and partial regularity of the suitable weak solutions
of the Navier--Stokes equations were obtained (see, \cite{Sch2}, \cite%
{CafKohNir}, \cite{Lin1}, \cite{Lio1}, \cite{Lad1}). There exist many works
which study different properties of solutions of the Navier--Stokes equation
(see, for example, \cite{Lio1}, \cite{Lad1}, \cite{Lin1}, \cite{Fef1}, \cite%
{FoiManRosTem}, \cite{FoiRosTem1}, \cite{FoiRosTem2}, \cite{FoiRosTem3}, 
\cite{GlaSveVic}, \cite{HuaWan}, \cite{PerZat}, \cite{Sol1}, \cite{Sol2}, 
\cite{Tem1}), etc.) and also different modifications of Navier--Stokes
equation (see, for example, \cite{Lio1}, \cite{Sol3}, etc.).

In this article an investigation of the question on uniqueness of the weak
solutions of the mixed problem with Dirichlet boundary condition for the
incompressible Navier-Stokes in the $3D$ case is given. Here for
investigation we use an approach that is different from usual methods which
are used for investigation of the questions of such type. Precisely this
approach allows us to solve the posed problem. So with the use of the this
approach we prove the uniqueness of the weak solutions of the incompressible
Navier--Stokes equations without complementary conditions on the velocity,
but under the complementary assumption on $h$ and $u_{0}$. And also we study
the auxiliary problems, more exactly we prove the existence and uniqueness
of the weak solutions of auxiliary problems. The main result of this article
is the following theorem:

\begin{theorem}
\label{Th_1}Let $\Omega \subset R^{3}$ be a bounded domain with sufficiently
smooth boundary $\partial \Omega $, $T>0$ be a number. Then for each given $%
\left( h,u_{0}\right) \in L^{2}\left( 0,T;\left( H^{1}\left( \Omega \right)
\right) ^{3}\right) \times $ $V\left( \Omega \right) $ and every fixed $p\in
L^{2}\left( 0,T;H^{1}\left( \Omega \right) \right) $ the incompressible $3D-$
Navier-Stokes Equations (i.e. problem (\ref{1}) - (\ref{3}) in $d=3$) has a
unique solution in $V\left( Q^{T}\right) $.
\end{theorem}

\section{\label{Sec_2}Preliminary results}

In the beginning we explore some properties that is connected with
uniqueness of solutions of the Navier--Stokes equations. As is well known
(see, for example, \cite{Lio1} and references therein) problem (\ref{1}) - (%
\ref{3}) possesses weak solution in the space $V\left( Q^{T}\right) $, that
will be defined later on, for any $u_{0i}\left( x\right) ,$ $h_{i}(x,t)$ ($i=%
\overline{1,3}$) which are contained in the suitable spaces (in the case $%
d=3 $, that we will investigate here, essentially).

\begin{definition}
\label{D_2.1}Let $V\left( Q^{T}\right) $ be the space determined as (see, 
\cite{Lio1}) 
\begin{equation*}
V\left( Q^{T}\right) \equiv L^{2}\left( 0,T;V\left( \Omega \right) \right) \cap W^{1,2}\left( 0,T;\left( H^{-1}\left( \Omega \right) \right) ^{3}\right) \cap L^{\infty }\left( 0,T;\left( H\left( \Omega \right) \right) ^{3}\right) ,%
\end{equation*}%
where 
\begin{equation*}
V\left( \Omega \right) =\left\{ v\left\vert \ v\in \right. \left( W_{0}^{1,2}\left( \Omega \right) \right) ^{3}\equiv \left( H_{0}^{1}\left( \Omega \right) \right) ^{3},\quad \func{div}v=0\right\} ,%
\end{equation*}%
and 
\begin{equation*}
u_{0}\in \left( H\left( \Omega \right) \right) ^{3},\quad h\in L^{2}\left(
0,T;\left( H^{-1}\left( \Omega \right) \right) ^{3}\right) ,
\end{equation*}%
here $\left( H\left( \Omega \right) \right) ^{3}$ is the closure in $\left(
L^{2}\left( \Omega \right) \right) ^{3}$ of 
\begin{equation*}
\left\{ \varphi \left\vert \ \varphi \in \left( C_{0}^{\infty }\left( \Omega
\right) \right) ^{3}\right. ,\func{div}\varphi =0\right\} .
\end{equation*}
\end{definition}

\begin{definition}
\label{D_2.2}A $u\in V\left( Q^{T}\right) $ is called the solution of
problem (\ref{1}) - (\ref{3}) if $u\left( t,x\right) $ satisfies the initial
condition $u\left( 0,x\right) =u_{0}\left( x\right) $ and the following
equation 
\begin{equation*}
\frac{d}{dt}\left\langle u,v\right\rangle -\left\langle \nu \Delta
u,v\right\rangle +\left\langle \underset{j=1}{\overset{d}{\sum }}u_{j}\frac{%
\partial u}{\partial x_{j}},v\right\rangle =\left\langle h-\nabla
p,v\right\rangle
\end{equation*}%
for any $v\in V\left( Q^{T}\right) $ and of every fixed $p\in L^{2}\left(
Q^{T}\right) $ on $\left( 0,T\right) $ in the sense of $L^{2}$.\footnote{%
For the widened explenation one can look \cite{Lio1}.}
\end{definition}

So, from now on we will use this definition together with the standard
notation that is widely used in the literature. Let the posed problem have
two different solutions $u(t,x),v(t,x)\in V\left( Q^{T}\right) $, then
within the known approach we get the function $w(t,x)=u(t,x)-v(t,x)$ of the
following problem (as is well known, if this method is used then the
pressure p "will disappear") 
\begin{equation}
\frac{1}{2}\frac{\partial }{\partial t}\left\Vert w\right\Vert _{2}^{2}+\nu \left\Vert \nabla w\right\Vert _{2}^{2}+\underset{j,k=1}{\overset{3}{\sum }}\left\langle \frac{\partial v_{k}}{\partial x_{j}}w_{k},w_{j}\right\rangle =0,
\label{2.1}
\end{equation}%
\begin{equation}
w\left( 0,x\right) \equiv w_{0}\left( x\right) =0,\quad x\in \Omega ;\quad w\left\vert \ _{\left( 0,T\right) \times \partial \Omega }\right. =0,
\label{2.2}
\end{equation}%
where $\left\langle f,g\right\rangle =\underset{i=1}{\overset{3}{\sum }}%
\underset{\Omega }{\int }f_{i}g_{i}dx$ for any $f,g\in \left( H\left( \Omega
\right) \right) ^{3}$, or $f\in \left( H^{-1}\left( \Omega \right) \right)
^{3}$ and $g\in \left( H^{-1}\left( \Omega \right) \right) ^{3}$,
respectively.

Now we have some remarks about properties of solutions of the problem (\ref%
{1}) - (\ref{3}). As is known (\cite{Ler1}, \cite{Lad1}, \cite{Lio1}),
problem (\ref{1}) - (\ref{3}) is solvable and possesses weak solution that
is contained in the space $V\left( Q^{T}\right) $, which is defined in
Definition \ref{D_2.1}. Therefore we will conduct our study under the
condition that problem (\ref{1}) - (\ref{3}) have weak solutions and they
are contained in $V\left( Q^{T}\right) $. For the study of the uniqueness of
the posed problem in the three dimensioned case we will use the ordinary
approach by assuming that problem (\ref{1}) - (\ref{3}) has, at least, two
different solutions $u(t,x),v(t,x)\in V\left( Q^{T}\right) $ but using a
different procedure we will demonstrate that this is not possible.

Consequently if we assume that problem (\ref{1}) - (\ref{3}) have two
different solutions then they need to be different at least on some
subdomain $Q_{1}^{T}$ of $Q^{T}$. In other words there exist a subdomain $%
\Omega _{1}$ of $\Omega $ and an interval $\left( t_{1},t_{2}\right)
\subseteq \left( 0,T\right] $ such that subdomain $Q_{1}^{T}$\ one can
define as $Q_{1}^{T}\subseteq \left( t_{1},t_{2}\right) \times \Omega
_{1}\subseteq Q^{T}$ with $meas_{4}\left( Q_{1}^{T}\right) >0$ for which the
following is true 
\begin{equation}
meas_{4}\left( \left\{ (t,x)\in Q^{T}\left\vert \ \left\vert u(t,x)-v(t,x)\right\vert \right. >0\right\} \right) =meas_{4}\left( Q_{1}^{T}\right) >0
\label{2.3}
\end{equation}%
here we denote the measure of $Q_{1}^{T}$ in $R^{4}$ as $meas_{4}\left(
Q_{1}^{T}\right) $ (Four dimensional Lebesgue measure.) Whence follows, that
for the subdomain $\Omega _{1}$ takes place the inequation: $meas_{3}(\Omega
_{1})>0$.

Even though we prove the following lemmas for $d>1$, we will use them mostly
for the case $d=4$.

In the beginning we prove the following lemmas that we will use later on.

\begin{lemma}
\label{L_2.1}Let $G\subset R^{d}$ be Lebesgue measurable subset then the
following statements are equivalent:

1) $\infty >meas_{d}\left( G\right) >0;$

2) there exist a subset $I\subset R^{1}$, $meas_{1}\left( I\right) >0$ and $%
G_{\beta }\subset L_{\beta ,d-1}$, $meas_{d-1}\left( G_{\beta }\right) >0$
such that $G=\underset{\beta \in I}{\cup }G_{\beta }\cup N$, where $N$ is a
set with $meas_{d-1}\left( N\right) =0$, and $L_{\beta ,d-1}$ is the
hyperplane of $R^{d}$, with $co\dim _{d}L_{\beta ,d-1}=1$, for any $\beta
\in I$, which is generated by single vector $y_{0}\in R^{d}$ and defined in
the following form 
\begin{equation*}
L_{\beta ,d-1}\equiv \left\{ y\in R^{d}\left\vert \ \left\langle
y_{0},y\right\rangle =\beta \right. \right\} ,\quad \forall \beta \in I.
\end{equation*}
\end{lemma}

\begin{proof}
Let $meas_{d}\left( G\right) >0$ and consider the class of hyperlanes $%
L_{\gamma ,d-1}$ for which $G\cap L_{\gamma ,d-1}\neq \varnothing $ and $%
\gamma \in I_{1}$, here $I_{1}\subset R^{1}$\ be some subset. It is clear
that 
\begin{equation*}
G\equiv \underset{\gamma \in I_{1}}{\bigcup }\left\{ x\in G\cap L_{\gamma
,d-1}\left\vert \ \gamma \in I_{1}\right. \right\} .
\end{equation*}%
Then there exists a subclass of hyperplanes $\left\{ L_{\gamma
,d-1}\left\vert \ \gamma \in I_{1}\right. \right\} $ for which the
inequality $meas_{d-1}\left( G\cap L_{\gamma ,d-1}\right) >0$ is satisfied.
The number of such type hyperplanes cannot be less than countable or equal
it because $meas_{d}\left( G\right) >0$, moreover this subclass of $I_{1}$\
must possess the $R^{1}$ measure greater than $0$ since $meas_{d}\left(
G\right) >0$. Indeed, let $I_{1,0}$ be this subclass and $meas_{1}\left(
I_{1,0}\right) =0$. If we consider the set 
\begin{equation*}
\left\{ \left( \gamma ,y\right) \in I_{1,0}\times G\cap L_{\gamma
,d-1}\left\vert \ \gamma \in I_{1,0},y\in G\cap L_{\gamma ,d-1}\right.
\right\} \subset R^{d}
\end{equation*}%
where $meas_{d-1}\left( G\cap L_{\gamma ,d-1}\right) >0$ for all $\gamma \in
I_{1,0}$, but $meas_{1}\left( I_{1,0}\right) =0$, then 
\begin{equation*}
meas_{d}\left( \left\{ \left( \gamma ,y\right) \in I_{1,0}\times G\cap
L_{\gamma ,d-1}\left\vert \ \gamma \in I_{1,0}\right. \right\} \right) =0.
\end{equation*}%
On the other hand we have 
\begin{equation*}
0=meas_{d}\left( \left\{ \left( \gamma ,y\right) \in I_{1}\times G\cap
L_{\gamma ,d-1}\left\vert \ \gamma \in I_{1}\right. \right\} \right)
=meas_{d}\left( G\right)
\end{equation*}%
as $meas_{d-1}\left( G\cap L_{\gamma ,d-1}\right) =0$ for all $\gamma \in
I_{1}-I_{1,0}$. But this contradicts the condition $meas_{d}\left( G\right)
>0$. Consequently, the statement 2 holds.

Let the statement 2 holds. It is clear that the class of hyperplanes $%
L_{\beta ,d-1}$ defined by such way are paralell and also we can define the
class of subsets of $G$ as its cross-section with hyperplanes, i.e. in the
form: $G_{\beta }\equiv G\cap L_{\beta ,d-1}$, $\beta \in I$. Then $G_{\beta
}\neq \varnothing $ and we can write $G_{\beta }\equiv G\cap L_{\beta ,d-1}$%
, $\beta \in I$, moreover $G\equiv \underset{\beta \in I}{\bigcup }\left\{
x\in G\cap L_{\beta ,d-1}\left\vert \ \beta \in I\right. \right\} \cup N$.
Whence we get 
\begin{equation*}
G\equiv \left\{ \left( \beta ,x\right) \in I\times G\cap L_{\beta
,d-1}\left\vert \ \beta \in I,x\in G\cap L_{\beta ,d-1}\right. \right\} \cup
N.
\end{equation*}

Consequently $meas_{d}\left( G\right) >0$ by virtue of conditions $%
meas_{1}\left( I\right) >0$ and $\ $

$meas_{d-1}\left( G_{\beta }\right) >0$ for any $\beta \in I$.
\end{proof}

From Lemma \ref{L_2.1} it follows that for the study of the measure of some
subset $%
\Omega
\subseteq R^{d}$ it is enough to study its foliations by a class of suitable
hyperplanes.

\begin{lemma}
\label{L_2.2}Let problem (\ref{1}) - (\ref{3}) has, at least, two different
solutions $u,v$ that are contained in $V\left( Q^{T}\right) $. Then there
exists, at least, one class of parallel and different hyperplanes $L_{\alpha
}$, $\alpha \in I\subseteq \left( \alpha _{1},\alpha _{2}\right) \subset
R^{1}$ ($\alpha _{2}>\alpha _{1}$)\ with $co\dim _{R^{3}}L_{\alpha }=1$
such, that $u\neq v$ on $Q_{L_{\alpha }}^{T}\equiv \left[ \left( 0,T\right)
\times \left( \Omega \cap L_{\alpha }\right) \right] \cap Q_{1}^{T}$, and
vice versa, here $meas_{1}\left( I\right) >0$ and $L_{\alpha }$ are
hyperlanes which are defined as follows there is vector $x_{0}\in
S_{1}^{R^{3}}\left( 0\right) $ such that 
\begin{equation*}
L_{\alpha }\equiv \left\{ x\in R^{3}\left\vert \ \left\langle
x_{0},x\right\rangle =\alpha ,\right. \ \forall \alpha \in I\right\} .
\end{equation*}
\end{lemma}

\begin{proof}
Let problem (\ref{1}) - (\ref{3}) have two different solutions $u,v\in
V\left( Q^{T}\right) $ then there exist a subdomain of $Q^{T}$ on which
these solutions are different. Then there are $t_{1},t_{2}>0$ such, that for
any $t\in J\subseteq \left[ t_{1},t_{2}\right] \subseteq \left[ 0,T\right) $
the following holds 
\begin{equation}
meas_{R^{3}}\left( \left\{ x\in \Omega \left\vert \ \left\vert u\left( t,x\right) -v\left( t,x\right) \right\vert >0\right. \right\} \right) >0
\label{2.4}
\end{equation}%
where $meas_{1}\left( J\right) >0$ by the virtue of the codition 
\begin{equation*}
meas_{4}\left( \left\{ (t,x)\in Q^{T}\left\vert \ \left\vert
u(t,x)-v(t,x)\right\vert \right. >0\right\} \right) >0
\end{equation*}%
and of Lemma \ref{L_2.1}. Hence follows, that there exist, at least, one
class of parallel hyperplanes $L_{\alpha }$, $\alpha \in I\subseteq \left(
\alpha _{1},\alpha _{2}\right) \subset R^{1}$ with $co\dim _{R^{3}}L_{\alpha
}=1$ such that 
\begin{equation}
meas_{R^{2}}\left\{ x\in \Omega \cap L_{\alpha }\left\vert \ \left\vert u\left( t,x\right) -v\left( t,x\right) \right\vert >0\right. \right\} >0,\ \forall \alpha \in I
\label{2.5}
\end{equation}%
for $\forall t\in J$, where the subset $I$ is such that $I\subseteq \left(
\alpha _{1},\alpha _{2}\right) \subset R^{1}$ with $meas_{1}\left( I\right)
>0$, $meas_{1}\left( J\right) >0$ and (\ref{2.5}) holds, by virtue of (\ref%
{2.4}). This proves the "if" part of Lemma.

Now consider the converse assertion. Let there exist a class of hyperplanes $%
L_{\alpha }$, $\alpha \in I_{1}\subseteq \left( \alpha _{1},\alpha
_{2}\right) \subset R^{1}$ with $co\dim _{R^{3}}L_{\alpha }=1$ that fulfills
the condition of Lemma and $I_{1}$\ satisfies the same condition $I$. Then
there exist, at least, one subset $J_{1}$ of $\left[ 0,T\right) $ such, that 
$meas_{1}\left( J_{1}\right) >0$ and the inequality $u\left( t,x\right) \neq
v\left( t,x\right) $ on $Q_{2}^{T}$ with $meas_{4}\left( Q_{2}^{T}\right) >0$
defined as $Q_{2}^{T}\equiv J_{1}\times U_{L}$ takes place, where 
\begin{equation}
U_{L}\equiv \underset{\alpha \in I_{1}}{\bigcup }\left\{ x\in \Omega \cap L_{\alpha }\left\vert \ u\left( t,x\right) \neq v\left( t,x\right) \right. \right\} \subset \Omega ,\ t\in J_{1}
\label{2.6}
\end{equation}%
for which the inequality $mes_{R^{3}}\left( U_{L}\right) >0$ is satisfied by
the condition and of Lemma \ref{L_2.1}.

So we get 
\begin{equation*}
u\left( t,x\right) \neq v\left( t,x\right) \quad \text{on \ }Q_{2}^{T}\equiv
J_{1}\times U_{L},\quad meas_{4}\left( Q_{2}^{T}\right) >0.
\end{equation*}%
Thus the fact that $u\left( t,x\right) $ and $v\left( t,x\right) $ are
different functions in $V\left( Q^{T}\right) $ follows.
\end{proof}

May be one can prove more general lemmas of such type with the use of
regularity properties of weak solutions of this problem (see, \cite{Sch1}, 
\cite{CafKohNir}, \cite{Lin1}, etc.).

\section{\label{Sec_3}Uniqueness of Solutions of Navier-Stokes Equations in
Three Dimension case}

From Lemma \ref{L_2.2} it follows that for the investigation of the posed
question it is enough to investigate this problem on the suitable
cross-sections of the domain $Q^{T}\equiv \left( 0,T\right) \times \Omega $.

So, firstly we will define subdomains of $Q^{T}\equiv \left( 0,T\right)
\times \Omega $ as follows $Q_{L}^{T}\equiv \left( 0,T\right) \times \left(
\Omega \cap L\right) $, where $L$ is arbitrary fixed hyperplane of the
dimension two and $\Omega \cap L\neq \varnothing $. Therefore we will study
the problem on the subdomain defined by the use of the cross-section of $%
\Omega $ by arbitrary fixed hyperplane dimension two $L$, i.e. by the $%
co\dim _{R^{3}}L=1$ ($\Omega \cap L$, namely on $Q_{L}^{T}\equiv \left(
0,T\right) \times \left( \Omega \cap L\right) $).

Consequently, we will investigate uniqueness of the problem (\ref{1}) - (\ref%
{3}) on the "cross-section" $Q^{T}$ defined by the cross-section of $\Omega $%
, where $\Omega \subset R^{3}$. This cross-section we understand in the
following sense: Let $L$ be a hyperplane in $R^{3}$, i.e. with $co\dim
_{R^{3}}L=1$, that is equivalent to $R^{2}$. We denote by $\Omega _{L}$ the
cross-section of the form $\Omega _{L}\equiv \Omega \cap L\neq \varnothing $%
, $mes_{R^{2}}\left( \Omega _{L}\right) >0$, in the particular case $L\equiv
\left( x_{1},x_{2},0\right) $. In the other words, if $L$ is the hyperplane
in $R^{3}$ then we can determine it as $L\equiv \left\{ x\in R^{3}\left\vert
\ a_{1}x_{1}+a_{2}x_{2}+a_{3}x_{3}=b\right. \right\} $, where coefficients $%
a_{i},b\in R^{1}$ ($i=1,2,3)$ are the arbitrary fixed constants. Whence
follows, that $a_{3}x_{3}=b-a_{1}x_{1}-a_{2}x_{2}$ or $x_{3}=\frac{1}{a_{3}}%
\left( b-a_{1}x_{1}-a_{2}x_{2}\right) $ if one assume $a_{i}\neq 0$ ($%
i=1,2,3 $), or if one takes into account of substitutions: $\frac{b}{a_{3}}%
\Longrightarrow b,\frac{a_{1}}{a_{3}}\Longrightarrow a_{1}$ and $\frac{a_{2}%
}{a_{3}}\Longrightarrow a_{2}$ we derive $x_{3}=b-a_{1}x_{1}-a_{2}x_{2}$ in
the new coefficients.

Thus we have 
\begin{equation}
D_{3}\equiv \frac{\partial x_{1}}{\partial x_{3}}D_{1}+\frac{\partial x_{2}}{\partial x_{3}}D_{2}=-a_{1}^{-1}D_{1}-a_{2}^{-1}D_{2}\quad \&
\label{3.1}
\end{equation}%
\begin{equation}
D_{3}^{2}=a_{1}^{-2}D_{1}^{2}+a_{2}^{-2}D_{2}^{2}+2a_{1}^{-1}a_{2}^{-1}D_{1}D_{2},\quad D_{i}=\frac{\partial }{\partial x_{i}},i=1,2,3.
\label{3.2}
\end{equation}

For the application of our approach we need to assume that functions $u_{0}$
and $h$\ posseses some smoothness. Moreover as is known from the existence
result $p$ is arbitrary fixed elements of $L^{2}\left( Q^{T}\right) $, but
we will assume here and its smoothness.

So we assume the following conditions in order to apply the our approach to
the posed problem, i.e. now we need take account the following condition of
Theorem \ref{Th_1} holds. More exactly:

Let $p\in L^{2}\left( 0,T;H^{1}\left( \Omega \right) \right) $ and 
\begin{equation*}
u_{0}\in \left( H_{0}^{1}\left( \Omega \right) \right) ^{3},\quad h\in
L^{2}\left( 0,T;\left( H^{1}\left( \Omega \right) \right) ^{3}\right) .
\end{equation*}

Then we can transform of problem (\ref{1}) - (\ref{3}) to the following
problem, that is equivalent to the posed problem on $\left[ 0,T\right)
\times \Omega _{L}$ by virtue of the above condition of the main theorem,
here $T>0$ some number, 
\begin{equation*}
\frac{\partial u}{\partial t}-\nu \Delta u+\underset{j=1}{\overset{3}{\sum }}%
u_{j}D_{j}u+\nabla p=\frac{\partial u_{L}}{\partial t}-\nu \left(
D_{1}^{2}+D_{2}^{2}+D_{3}^{2}\right) u_{L}+
\end{equation*}%
\begin{equation*}
u_{L1}D_{1}u_{L}+u_{L2}D_{2}u_{L}+u_{L3}D_{3}u_{L}+\nabla p_{L}=\frac{%
\partial u_{L}}{\partial t}-\nu \left[
D_{1}^{2}+D_{2}^{2}+a_{1}^{-2}D_{1}^{2}\right. +
\end{equation*}%
\begin{equation*}
\left. a_{2}^{-2}D_{2}^{2}+2a_{1}^{-1}a_{2}^{-1}D_{1}D_{2}\right]
u_{L}+u_{L1}D_{1}u_{L}+u_{L2}D_{2}u_{L}-u_{L3}a_{1}^{-1}D_{1}u_{L}-
\end{equation*}%
\begin{equation*}
u_{L3}a_{2}^{-1}D_{2}u_{L}+\nabla p_{L}=\frac{\partial u_{L}}{\partial t}%
-\nu \left[ \left( 1+a_{1}^{-2}\right) D_{1}^{2}+\left( 1+a_{2}^{-2}\right)
D_{2}^{2}\right] u_{L}-
\end{equation*}%
\begin{equation}
2\nu a_{1}^{-1}a_{2}^{-1}D_{1}D_{2}u_{L}+\left( u_{L1}-a_{1u_{L}3}^{-1}\right) D_{1}u_{L}+\left( u_{L2}-a_{2}^{-1}u_{L3}\right) D_{2}u_{L}+\nabla p_{L}=h_{L}
\label{3.3}
\end{equation}%
on $\left( 0,T\right) \times \Omega _{L}$, by virtue of (\ref{3.1}) and (\ref%
{3.2}). We get 
\begin{equation}
\func{div}u_{L}=D_{1}\left( u_{L}-a_{1}^{-1}u_{L3}\right) +D_{2}\left( u_{L}-a_{2}^{-1}u_{L3}\right) =0,\quad x\in \Omega _{L},\ t>0
\label{3.4}
\end{equation}%
\begin{equation}
u_{L}\left( 0,x\right) =u_{L0}\left( x\right) ,\quad \left( t,x\right) \in \left( 0,T\right) \times \Omega _{L};\quad u_{L}\left\vert \ _{\left( 0,T\right) \times \partial \Omega _{L}}\right. =0.
\label{3.5}
\end{equation}%
using same way.

In the beginning it is necessary to investigate the existence of the
solution of problem (\ref{3.3}) - (\ref{3.5}) and determine the space where
the existing solutions are contained. Consequently for ending the proof of
the uniqueness theorem, it is enough to prove the existence theorem and the
uniqueness theorem for the derived problem (\ref{3.3}) - (\ref{3.5}), in
this case. So now we will investigate (\ref{3.3}) - (\ref{3.5}).

\subsection{\label{Subsec_3.1}Existence of Solution of Problem (\protect\ref%
{3.3}) - (\protect\ref{3.5}).}

To carry out the known argument started by Leray (\cite{Ler1}, see, also 
\cite{Lio1}) we can determine the following space 
\begin{equation*}
V\left( \Omega _{L}\right) =\left\{ v\left\vert \ v\in \right. \left(
W_{0}^{1,2}\left( \Omega _{L}\right) \right) ^{3}\equiv \left(
H_{0}^{1}\left( \Omega _{L}\right) \right) ^{3}\quad ,\func{div}v=0\right\} ,
\end{equation*}%
where $\func{div}$ is regarded in the sense (\ref{3.4}). Consequently, a
solution of this problem will be understood as follows: Let $h_{L}\in
L^{2}\left( 0,T;\left( H^{-1}\left( \Omega _{L}\right) \right) ^{3}\right) $
and $u_{0L}\in \left( H_{0}\left( \Omega _{L}\right) \right) ^{3}$, here $%
\left( H\left( \Omega _{L}\right) \right) ^{3}$ is the closure in $\left(
L^{2}\left( \Omega _{L}\right) \right) ^{3}$ of 
\begin{equation*}
\left\{ \varphi \left\vert \ \varphi \in \left( C_{0}^{\infty }\left( \Omega
_{L}\right) \right) ^{3}\right. ,\func{div}\varphi =0\right\} .
\end{equation*}

So, following the terminology used by J.-L. Lions \cite{Lio1} we call the
solutions of the problem (\ref{3.3}) - (\ref{3.5}) a pair of functions $%
\left( u_{L}(t,x),p_{L}(t,x)\right) $ if $\left(
u_{L}(t,x),p_{L}(t,x)\right) $ is a solution of the problem 
\begin{equation*}
\frac{d}{dt}\left\langle u_{L},v\right\rangle -\left\langle \nu \Delta
u_{L},v\right\rangle +\left\langle \underset{j=1}{\overset{3}{\sum }}%
u_{Lj}D_{j}u_{L},v\right\rangle =\left\langle h_{L}-\nabla
p_{L},v\right\rangle ,\quad \left\langle u_{L}\left( x\right)
,v\right\rangle =\left\langle u_{0L},v\right\rangle ,
\end{equation*}%
for any $v\in V\left( \Omega _{L}\right) $, then a function $p_{L}\left(
t,x\right) $ will be chosen as a fixed element of $L^{2}\left( \left(
0,T\right) \times \Omega _{L}\right) \equiv L^{2}\left( Q_{L}^{T}\right) $,
here $\left\langle \circ ,\circ \right\rangle $ is the dual form for the
pair of spaces $\left( V\left( \Omega _{L}\right) ,\left( H^{-1}\left(
\Omega _{L}\right) \right) ^{3}\right) $.

Hence we obtain 
\begin{equation*}
\frac{1}{2}\frac{d}{dt}\left\Vert u_{L}\right\Vert _{\left( H\left( \Omega
_{L}\right) \right) ^{3}}^{2}\left( t\right) +\nu \left( 1+a_{1}^{-2}\right)
\left\Vert D_{1}u_{L}\right\Vert _{\left( H\left( \Omega _{L}\right) \right)
^{3}}^{2}\left( t\right) +
\end{equation*}%
\begin{equation*}
\nu \left( 1+a_{2}^{-2}\right) \left\Vert D_{2}u_{L}\right\Vert _{\left(
H\left( \Omega _{L}\right) \right) ^{3}}^{2}\left( t\right) +2\nu
a_{1}^{-1}a_{2}^{-1}\left\langle D_{1}u_{L},D_{2}u_{L}\right\rangle \left(
t\right) =\left\langle h_{L},u_{L}\right\rangle ,
\end{equation*}%
with use of (\ref{3.1}) and next (\ref{3.4}), where $\left\langle
f,g\right\rangle =\underset{i=1}{\overset{3}{\sum }}\underset{\Omega _{L}}{%
\int }f_{i}g_{i}dx$ for any $f,g\in \left( H\left( \Omega _{L}\right)
\right) ^{3}$, or $f\in \left( H^{1}\left( \Omega _{L}\right) \right) ^{3}$
and $g\in \left( H^{-1}\left( \Omega _{L}\right) \right) ^{3}$, respectively.

From the above equality by usual calculations (as in \cite{Lio1}, \cite%
{SolAhm}, etc.) we get a priori estimates for the functions $u\left(
t,x\right) $ that shows the inclusion $u_{L}\in V\left( Q_{L}^{T}\right) $,
where

\begin{equation}
V\left( Q_{L}^{T}\right) \equiv L^{2}\left( 0,T;V\left( \Omega _{L}\right) \right) \cap W^{1,2}\left( 0,T;\left( H^{-1}\left( \Omega _{L}\right) \right) ^{3}\right) \cap L^{\infty }\left( 0,T;\left( H\left( \Omega _{L}\right) \right) ^{3}\right) .
\label{3.6}
\end{equation}

If one take into account the stationary part (or elliptic part) in the left
side of above equality then it is not difficult to see the coerciveness of
the operator induced by this part from $L^{2}\left( 0,T;V\left( \Omega
_{L}\right) \right) $ to $L^{2}\left( 0,T;\left( H^{-1}\left( \Omega
_{L}\right) \right) ^{3}\right) $. Moreover with the use of the embedding
theorems (see, \cite{Lio1}, \cite{Sol3}, \cite{Sol2}) we obtain the weak
compactness of the operator induced by the posed problem, also. The
calculations of such type were used in many works devoted to the problems of
such type(see, in particular, \cite{Lio1}, \cite{SolAhm}, \cite{Sol4} and
their references).

So, with use of methods employed for problems of such type (see, for
example, \cite{Lio1}, \cite{Sol1} etc.) we obtain solvability of this
problem in the space $V\left( Q_{L}^{T}\right) $.

Consequently the following result is proved.

\begin{theorem}
\label{Th_2.1}Under above conditions for any $\left( h_{L},u_{0L}\right) \in
L^{2}\left( 0,T;\left( H^{-1}\left( \Omega _{L}\right) \right) ^{3}\right)
\times $ $\left( H\left( \Omega _{L}\right) \right) ^{3}$ problem (\ref{3.3}%
) - (\ref{3.5}) has weak solutions $\left( u_{L}\left( t,x\right)
,p_{L}\left( t,x\right) \right) $ that is contained in $V_{L}\left(
Q_{L}^{T}\right) \times $ $L_{2}\left( Q_{L}^{T}\right) $, here $p_{L}\in
L_{2}\left( Q_{L}^{T}\right) $ is arbitrary fixed element.
\end{theorem}

We need to note that for the proof of this theorem it is enough to apply the
known general solvability result from \cite{Sol3} (or \cite{Sol2}, \cite%
{Sol4}, \cite{Tem1} see, also their references).

\subsection{\label{Subsec_3.2}Uniqueness of Solution of Problem (\protect\ref%
{3.3}) - (\protect\ref{3.5}).}

For the study of the uniqueness of the solution as usually: we will assume
that the posed problem have two different solutions $u=\left(
u_{1},u_{2},u_{3}\right) $, $v=\left( v_{1},v_{2},v_{3}\right) $ and\ \ $%
p_{1},p_{2}$, and we will investigate its difference: $w=u-v$, $%
p=p_{1}-p_{2} $.\ (Here for brevity we won't specify indexes for functions
as we investigate problem (\ref{3.3}) - (\ref{3.5}) on $Q_{L}^{T}$.) Then
for $w,p$ we obtain the following problem 
\begin{equation*}
\frac{\partial w}{\partial t}-\nu \left[ \left( 1+a_{1}^{-2}\right)
D_{1}^{2}+\left( 1+a_{2}^{-2}\right) D_{2}^{2}\right] w-2\nu
a_{1}^{-1}a_{2}^{-1}D_{1}D_{2}w+
\end{equation*}%
\begin{equation*}
\left( u_{1}-a_{1}^{-1}u_{3}\right) D_{1}u-\left(
v_{1}-a_{1}^{-1}v_{3}\right) D_{1}v+\left( u_{2}-a_{2}^{-1}u_{3}\right)
D_{2}u-
\end{equation*}%
\begin{equation*}
\left( v_{2}-a_{2}^{-1}v_{3}\right) D_{2}v+\nabla p=0,
\end{equation*}%
\begin{equation*}
\func{div}w=D_{1}\left[ \left( u-a_{1}^{-1}u_{3}\right) -\left(
v-a_{1}^{-1}v_{3}\right) \right] +D_{2}\left[ \left(
u-a_{2}^{-1}u_{3}\right) \right. -
\end{equation*}%
\begin{equation}
\left. \left( v-a_{2}^{-1}v_{3}\right) \right] =D_{1}w+D_{2}w-\left( a_{1}^{-1}D_{1}+a_{2}^{-1}D_{2}\right) w_{3}=0,
\label{3.7}
\end{equation}%
\begin{equation}
w\left( 0,x\right) =0,\quad x\in \Omega \cap L;\quad w\left\vert \ _{\left( 0,T\right) \times \partial \Omega _{L}}\right. =0.
\label{3.8}
\end{equation}

Hence we derive 
\begin{equation*}
\frac{1}{2}\frac{d}{dt}\left\Vert w\right\Vert _{2}^{2}+\nu \left[ \left(
1+a_{1}^{-2}\right) \left\Vert D_{1}w\right\Vert _{2}^{2}+\left(
1+a_{2}^{-2}\right) \left\Vert D_{2}w\right\Vert _{2}^{2}\right] +
\end{equation*}%
\begin{equation*}
2\nu a_{1}^{-1}a_{2}^{-1}\left\langle D_{1}w,D_{2}w\right\rangle
+\left\langle \left( u_{1}-a_{1}^{-1}u_{3}\right) D_{1}u-\left(
v_{1}-a_{1}^{-1}v_{3}\right) D_{1}v,w\right\rangle +
\end{equation*}%
\begin{equation*}
\left\langle \left( u_{2}-a_{2}^{-1}u_{3}\right) D_{2}u-\left(
v_{2}-a_{2}^{-1}v_{3}\right) D_{2}v,w\right\rangle =0
\end{equation*}%
or 
\begin{equation*}
\frac{1}{2}\frac{d}{dt}\left\Vert w\right\Vert _{2}^{2}+\nu \left(
\left\Vert D_{1}w\right\Vert _{2}^{2}+\left\Vert D_{2}w\right\Vert
_{2}^{2}\right) +\nu \left[ a_{1}^{-2}\left\Vert D_{1}w\right\Vert
_{2}^{2}+\right.
\end{equation*}%
\begin{equation*}
\left. a_{2}^{-2}\left\Vert D_{2}w\right\Vert
_{2}^{2}+2a_{1}^{-1}a_{2}^{-1}\left\langle D_{1}w,D_{2}w\right\rangle \right]
+\left\langle u_{1}D_{1}u-v_{1}D_{1}v,w\right\rangle +
\end{equation*}%
\begin{equation}
\left\langle u_{2}D_{2}u-v_{2}D_{2}v,w\right\rangle -a_{1}^{-1}\left\langle u_{3}D_{1}u-v_{3}D_{1}v,w\right\rangle -a_{2}^{-1}\left\langle u_{3}D_{2}u-v_{3}D_{2}v,w\right\rangle =0.
\label{3.9}
\end{equation}

If we consider the last 4 added elements of left part (\ref{3.9}),
separately, and if we simplify the calculations then we get 
\begin{equation*}
\left\langle w_{1}D_{1}u,w\right\rangle +\left\langle
v_{1}D_{1}w,w\right\rangle +\left\langle w_{2}D_{2}u,w\right\rangle
+\left\langle v_{2}D_{2}w,w\right\rangle -
\end{equation*}%
\begin{equation*}
a_{1}^{-1}\left\langle w_{3}D_{1}u,w\right\rangle -a_{1}^{-1}\left\langle
v_{3}D_{1}w,w\right\rangle -a_{2}^{-1}\left\langle
w_{3}D_{2}u,w\right\rangle -a_{2}^{-1}\left\langle
v_{3}D_{2}w,w\right\rangle =
\end{equation*}%
\begin{equation*}
\left\langle w_{1}D_{1}u,w\right\rangle +\frac{1}{2}\left\langle
v_{1},D_{1}w^{2}\right\rangle +\left\langle w_{2}D_{2}u,w\right\rangle +%
\frac{1}{2}\left\langle v_{2},D_{2}w^{2}\right\rangle
-a_{1}^{-1}\left\langle w_{3}D_{1}u,w\right\rangle -
\end{equation*}%
\begin{equation*}
\frac{1}{2}a_{1}^{-1}\left\langle v_{3},D_{1}w^{2}\right\rangle
-a_{2}^{-1}\left\langle w_{3}D_{2}u,w\right\rangle -\frac{1}{2}%
a_{2}^{-1}\left\langle v_{3},D_{2}w^{2}\right\rangle =
\end{equation*}%
\begin{equation*}
\frac{1}{2}\left\langle v_{1}-a_{1}^{-1}v_{3},D_{1}w^{2}\right\rangle +\frac{%
1}{2}\left\langle v_{2}-a_{2}^{-1}v_{3},D_{2}w^{2}\right\rangle +
\end{equation*}%
\begin{equation*}
\left\langle \left( w_{1}-a_{1}^{-1}w_{3}\right) w,D_{1}u\right\rangle
+\left\langle \left( w_{2}-a_{2}^{-1}w_{3}\right) w,D_{2}u\right\rangle =
\end{equation*}%
\begin{equation*}
\left\langle \left( w_{1}-a_{1}^{-1}w_{3}\right) w,D_{1}u\right\rangle
+\left\langle \left( w_{2}-a_{2}^{-1}w_{3}\right) w,D_{2}u\right\rangle .
\end{equation*}%
In the last equality we use the condition $\func{div}v=0$ (see, (\ref{3.4}))
and the condition (\ref{3.8}).

If we take into account this equality in equation (\ref{3.9}) then we get
the equation 
\begin{equation*}
\frac{1}{2}\frac{d}{dt}\left\Vert w\right\Vert _{2}^{2}+\nu \left(
\left\Vert D_{1}w\right\Vert _{2}^{2}+\left\Vert D_{2}w\right\Vert
_{2}^{2}\right) +\nu \left[ a_{1}^{-2}\left\Vert D_{1}w\right\Vert
_{2}^{2}+\right.
\end{equation*}%
\begin{equation*}
\left. a_{2}^{-2}\left\Vert D_{2}w\right\Vert
_{2}^{2}+2a_{1}^{-1}a_{2}^{-1}\left\langle D_{1}w,D_{2}w\right\rangle \right]
+\left\langle \left( w_{1}-a_{1}^{-1}w_{3}\right) w,D_{1}u\right\rangle +
\end{equation*}%
\begin{equation}
\left\langle \left( w_{2}-a_{2}^{-1}w_{3}\right) w,D_{2}u\right\rangle =0,\quad \left( t,x\right) \in \left( 0,T\right) \times \Omega _{L}
\label{3.10}
\end{equation}

Consequently, we derive the Cauchy problem for the equation (\ref{3.10})
with the initial condition 
\begin{equation}
\left\Vert w\right\Vert _{2}\left( 0\right) =0.  \label{3.11}
\end{equation}

Hence giving rise to the differential inequality we get the following Cauchy
problem for the differential inequality 
\begin{equation*}
\frac{1}{2}\frac{d}{dt}\left\Vert w\right\Vert _{2}^{2}+\nu \left(
\left\Vert D_{1}w\right\Vert _{2}^{2}+\left\Vert D_{2}w\right\Vert
_{2}^{2}\right) \leq \left\vert \left\langle \left(
w_{1}-a_{1}^{-1}w_{3}\right) w,D_{1}u\right\rangle \right\vert +
\end{equation*}%
\begin{equation}
\left\vert \left\langle \left( w_{2}-a_{2}^{-1}w_{3}\right) w,D_{2}u\right\rangle \right\vert ,
\label{3.10'}
\end{equation}%
with the initial condition (\ref{3.11}).

We have the following estimate for the right side of (3.10')%
\begin{equation*}
\left\vert \left\langle \left( w_{1}-a_{1}^{-1}w_{3}\right)
w,D_{1}u\right\rangle \right\vert +\left\vert \left\langle \left(
w_{2}-a_{2}^{-1}w_{3}\right) w,D_{2}u\right\rangle \right\vert \leq
\end{equation*}%
\begin{equation*}
\left( \left\Vert w_{1}-a_{1}^{-1}w_{3}\right\Vert _{4}+\left\Vert
w_{2}-a_{2}^{-1}w_{3}\right\Vert _{4}\right) \left\Vert w\right\Vert
_{4}\left\Vert \nabla u\right\Vert _{2}\leq
\end{equation*}%
whence with the use of Gagliardo-Nirenberg inequality (\cite{BesIlNik}) we
have 
\begin{equation*}
\left( 1+\max \left\{ \left\vert a_{1}^{-1}\right\vert ,\left\vert
a_{2}^{-1}\right\vert \right\} \right) \left\Vert w\right\Vert
_{4}^{2}\left\Vert \nabla u\right\Vert _{2}\leq c\left\Vert w\right\Vert
_{2}\left\Vert \nabla w\right\Vert _{2}\left\Vert \nabla u\right\Vert _{2}.
\end{equation*}%
It follows to note that 
\begin{equation*}
\left( w_{1}-a_{1}^{-1}w_{3}\right) w,\ \left( w_{2}-a_{2}^{-1}w_{3}\right)
w\in L^{2}\left( 0,T;V^{\ast }\left( \Omega _{L}\right) \right)
\end{equation*}%
by virtue of (\ref{3.6}).

Now taking this into account in (\ref{3.10'}) one can arrive the following
Cauchy problem for inequality 
\begin{equation*}
\frac{1}{2}\frac{d}{dt}\left\Vert w\right\Vert _{2}^{2}\left( t\right) +\nu
\left\Vert \nabla w\right\Vert _{2}^{2}\left( t\right) \leq c\left\Vert
w\right\Vert _{2}\left( t\right) \left\Vert \nabla w\right\Vert _{2}\left(
t\right) \left\Vert \nabla u\right\Vert _{2}\left( t\right) \leq
\end{equation*}%
\begin{equation*}
C\left( c,\nu \right) \left\Vert \nabla u\right\Vert _{2}^{2}\left( t\right)
\left\Vert w\right\Vert _{2}^{2}\left( t\right) +\nu \left\Vert \nabla
w\right\Vert _{2}^{2}\left( t\right) ,\quad \left\Vert w\right\Vert
_{2}\left( 0\right) =0
\end{equation*}%
since $w\in L^{\infty }\left( 0,T;\left( H\left( \Omega _{L}\right) \right)
^{3}\right) $, and consequently $\left\Vert w\right\Vert _{2}\left\Vert
\nabla w\right\Vert _{2}\in L^{2}\left( 0,T\right) $ by the virtue of the
above existence theorem $w\in V\left( Q_{L}^{T}\right) $, here $C\left(
c,\nu \right) >0$ is constant.

Thus we obtain the problem 
\begin{equation*}
\frac{d}{dt}\left\Vert w\right\Vert _{2}^{2}\left( t\right) \leq 2C\left(
c,\nu \right) \left\Vert \nabla u\right\Vert _{2}^{2}\left( t\right)
\left\Vert w\right\Vert _{2}^{2}\left( t\right) ,\quad \left\Vert
w\right\Vert _{2}\left( 0\right) =0,
\end{equation*}%
if we denote $\left\Vert w\right\Vert _{2}^{2}\left( t\right) \equiv y\left(
t\right) $ then 
\begin{equation*}
\frac{d}{dt}y\left( t\right) \leq 2C\left( c,\nu \right) \left\Vert \nabla
u\right\Vert _{2}^{2}\left( t\right) y\left( t\right) ,\quad y\left(
0\right) =0.
\end{equation*}

Consequently we obtain $\left\Vert w\right\Vert _{2}^{2}\left( t\right)
\equiv y\left( t\right) =0$, i.e. the following result is proved: \ 

\begin{theorem}
\label{Th_2.2}Under conditions of Lemma 2.1 for each given 
\begin{equation*}
\left( h,u_{0}\right) \in L^{2}\left( 0,T;\left( H^{-1}\left( \Omega
_{L}\right) \right) ^{3}\right) \times \left( H\left( \Omega _{L}\right)
\right) ^{3}
\end{equation*}
problem (\ref{3.3}) - (\ref{3.5}) has a unique weak solution $\left( u\left(
t,x\right) ,p\left( t,x\right) \right) $ that is contained in $V\left(
Q_{L}^{T}\right) \times $ $L_{2}\left( Q_{L}^{T}\right) $ (here as above $%
p\in L_{2}\left( Q_{L}^{T}\right) $ is arbitrary fixed element).
\end{theorem}

\subsection{\label{Subsec_3.3}Proof of Theorem \protect\ref{Th_1}}

\begin{proof}
(of Theorem \ref{Th_1}). As is known (\cite{Ler1}, \cite{Lad2}, \cite{Lio1}%
), problem (\ref{1}) - (\ref{3}) is solvable and possesses weak solution
that is contained in the space $V\left( Q^{T}\right) $. So, assume problem (%
\ref{1}) - (\ref{3}) has, at least, two different solutions under conditions
of Theorem \ref{Th_1}.

It is clear that if the problem have more than one solution then there is,
at least, some subdomain of $Q^{T}\equiv \left( 0,T\right) \times \Omega $,
on which this problem has, at least, two solutions such, that each from the
other are different. Consequently, starting from the above Lemma \ref{L_2.2}
we need to investigate the existence and uniqueness of the posed problem on
arbitrary fixed subdomain on which it is possibl that our problem can
possess more than one solution, more exactly in the case when the subdomain
is generated by an arbitrary fixed hyperplane by the virtue of Lemma \ref%
{L_2.2}. It is clear that, for us it is enough to prove that no such
subdomain generated by a hyperplane on which more than single solutions of
problem (\ref{1}) - (\ref{3}) exists, again by virtue of Lemma \ref{L_2.2}.
In other words, for us it remains to use the above results (i.e. Theorems %
\ref{Th_2.1} and \ref{Th_2.2}) in order to end the proof.

From the proved theorems above we obtain that there does not exist a
subdomain, defined in the previous section, on which problem (\ref{1}) - (%
\ref{3}) reduced on this subdomain might possesses more than one weak
solution. Consequently, taking Lemma \ref{L_2.2} into account we obtain that
the incompressible Navier-Stokes equation (i.e. problem (\ref{1}) - (\ref{3}%
)) under conditions of Theorem \ref{Th_1} possesses only one weak solution.
\end{proof}

Hence one can make the following conclusion

\section{Conclusion}

Let $p\in L^{2}\left( Q^{T}\right) $ is arbitrary fixed element and 
\begin{equation*}
\left( h,u_{0}\right) \in L^{2}\left( 0,T;\left( H^{-1}\left( \Omega \right)
\right) ^{3}\right) \times \left( H\left( \Omega \right) \right) ^{3}.
\end{equation*}%
Well known that the following inclusions are dense 
\begin{equation*}
L^{2}\left( 0,T;H^{1}\left( \Omega \right) \right) \subset L^{2}\left(
Q^{T}\right) ;\ V\left( \Omega \right) \subset \left( H\left( \Omega \right)
\right) ^{3}\ \ \&
\end{equation*}%
\begin{equation*}
L^{2}\left( 0,T;\left( H^{1}\left( \Omega \right) \right) ^{3}\right)
\subset L^{2}\left( 0,T;\left( H^{-1}\left( \Omega \right) \right)
^{3}\right)
\end{equation*}%
consequently there exist sequences 
\begin{equation*}
\left\{ p_{m}\right\} _{m=1}^{\infty }\subset L^{2}\left( 0,T;H^{1}\left(
\Omega \right) \right) ;\ \left\{ u_{0m}\right\} _{m=1}^{\infty }\subset
V\left( \Omega \right) ;
\end{equation*}%
\begin{equation*}
\ \left\{ h_{m}\right\} _{m=1}^{\infty }\subset L^{2}\left( 0,T;\left(
H^{1}\left( \Omega \right) \right) ^{3}\right)
\end{equation*}%
such that $p_{m}\longrightarrow p$ in $L^{2}\left( Q^{T}\right) $ and $%
\left\Vert p_{m}\right\Vert _{L^{2}\left( Q^{T}\right) }\leq \left\Vert
p\right\Vert _{L^{2}\left( Q^{T}\right) }$, $u_{0m}\longrightarrow u_{0}$ in 
$\left( H\left( \Omega \right) \right) ^{3}$ and $\left\Vert
u_{0m}\right\Vert _{\left( H\left( \Omega \right) \right) ^{3}}\leq
\left\Vert u_{0}\right\Vert _{\left( H\left( \Omega \right) \right) ^{3}}$, $%
h_{m}\longrightarrow h$ in $L^{2}\left( 0,T;\left( H^{-1}\left( \Omega
\right) \right) ^{3}\right) $ and $\left\Vert h_{m}\right\Vert _{L^{2}\left(
0,T;\left( H^{-1}\left( \Omega \right) \right) ^{3}\right) }\leq \left\Vert
h\right\Vert _{L^{2}\left( 0,T;\left( H^{-1}\left( \Omega \right) \right)
^{3}\right) }$.

Consequently for any $\varepsilon >0$ there exist $m\left( \varepsilon
\right) \geq 1$ such that under $m\geq m\left( \varepsilon \right) $ for the
corresponding elements $p_{m}$, $u_{0m}$, $h_{m}$ of the above sequences 
\begin{equation*}
\left\Vert p-p_{m}\right\Vert _{L^{2}\left( Q^{T}\right) }<\varepsilon ;\
\left\Vert u_{0}-u_{0m}\right\Vert _{\left( H\left( \Omega \right) \right)
^{3}}<\varepsilon ;\ \left\Vert h-h_{m}\right\Vert _{L^{2}\left( 0,T;\left(
H^{-1}\left( \Omega \right) \right) ^{3}\right) }<\varepsilon 
\end{equation*}%
hold, and also the claim of Theorem \ref{Th_1} is valid for problem (\ref{1}%
) - (\ref{3}) with these elements. 

One can note that the space that is everywhere dense subset of the necessary
space possess the minimal smoothnes in the relation with this space and also
is sufficient for the application of our approach. So we establish:

\begin{proposition}
For any $\left( h,u_{0},p\right) $ from the space $L^{2}\left(
0,T;H^{1}\left( \Omega \right) \right) \times V\left( \Omega \right) \times
L^{2}\left( 0,T;H^{1}\left( \Omega \right) \right) $ that is everywhere
dense in the space $L^{2}\left( 0,T;\left( H^{-1}\left( \Omega \right)
\right) ^{3}\right) \times \left( H\left( \Omega \right) \right) ^{3}\times
L^{2}\left( Q^{T}\right) $ the incompressible Navier-Stokes equation has
unique solution in $V\left( Q^{T}\right) $ in the sense of Definition \ref%
{D_2.2}.
\end{proposition}

\bigskip

\end{document}